\documentclass[11pt]{article}
\usepackage[latin1]{inputenc}
\usepackage{amsfonts}
\usepackage{mathrsfs,amssymb,amsmath,wasysym}
\usepackage{enumerate,textcomp}
\usepackage{amsthm}
\usepackage{pifont}
\usepackage[numbers,sort&compress]{natbib}
\usepackage{setspace}

\newtheorem{theorem}{Theorem}[section]
\newtheorem{lemma}[theorem]{Lemma}
\newtheorem{proposition}[theorem]{Proposition}
\newtheorem{definition}[theorem]{Definition}

\newcommand{\ve}{\varepsilon}

\usepackage{color}

\def \te  {\theta}

\def \n  {\nonumber}

\def \mc{\mathcal}
\def \ms{\mathscr}
\def \mf{\mathfrak}
\def \te{\theta}
\def \ra{\rightarrow}

\def \wt{\widetilde}
\def \mb{\mathbb}
\def \ve{\varepsilon}
\def \ov{\overline}

\title{\vspace{-0.in}\parbox{\linewidth }{\footnotesize\noindent
} \\  \bf Approximate controllability results for fractional semilinear integro-differential inclusions in Hilbert spaces}

\author{\\ N.I. Mahmudov\\ \small {Department of Mathematics, Eastern Mediterranean University, Gazimagusa, }\\ \small{T.R. North Cyprus, Mersin 10, Turkey. email: nazim.mahmudov@emu.edu.tr} \\
V. Vijayakumar\\ \small {Department of Mathematics,  Info Institute of Engineering, Kovilpalayam, }\\  \small{Coimbatore - 641 107, Tamil Nadu, India. email: vijaysarovel@gmail.com}\\
C. Ravichandran\\ \small {Department of Mathematics, KPR Institute of Engineering and Technology,}\\  \small{Arasur, Coimbatore - 641 407, \ Tamil Nadu, \ India, email: ravibirthday@gmail.com}\\
R. Murugesu \\ \small {Department of Mathematics,  SRMV College of Arts and Science,}\\ \small{Coimbatore - 641 020, \ Tamil Nadu, India. \ email: arjhunmurugesh@gmail.com}}
\date{ }

\begin{document}

 \maketitle \footnotetext[1]{ Corresponding author: N.I. Mahmudov}
\author{ \ }

\begin{abstract}
In this paper, we consider a class of fractional integro-differential inclusions in Hilbert spaces. This paper deals with the approximate controllability for a class of fractional integro-differential control systems. First, we establishes a set of sufficient conditions for the approximate controllability for a class of fractional semilinear integro-differential inclusions in Hilbert spaces. We use Bohnenblust-Karlin's fixed point theorem to prove our main results. Further, we extend the result to study the approximate controllability concept with nonlocal conditions. An example is also given to illustrate our main results.
\end{abstract}


{\bf 2010 Subject Classification:} 34A08, 34K40, 34G20.

\section{Introduction}
\noindent

Fractional order semilinear equations are abstract formulations for many problems arising in engineering and physics. The potential applications of fractional calculus are in diffusion process, electrical science, electrochemistry, viscoelasticity, control science, electro magnetic theory and several more. In fact, such models can be considered as an efficient alternative to the classical nonlinear differential models to simulate many complex processes. In the recent years, there has been a significant development in ordinary and partial differential equations involving fractional derivatives, see the monographs of Kilbas et al. \cite{aak1}, Lakshmikantham et al. \cite{vl3}, Miller and Ross \cite{ksm1}, Podlubny \cite{ip1}, Hilfer \cite{rh1}. On the other hand, the fractional differential inclusions arise in the mathematical modeling of certain problems in economics, optimal controls, etc., so the problem of existence of solutions of differential inclusions and fractional differential equations and differential inclusions has been studied by several authors for different kind of problems \cite{eb1, mb1, ec1, cc3, cue, cue1, jpc3, jpc5, jpc1, vv5, yz12}.

The concept of controllability is an important property of a control system which plays an important role in many control problems such as stabilization of unstable systems by feedback control. Therefore, in recent years controllability problems for various types of linear and nonlinear deterministic and stochastic dynamic systems have been studied in many publications (see \cite{ykc2, ykc4, jk2, jk3, nim1, cr3, cr4, vv1, vv3, vv4} and the references therein). From the mathematical point of view, the problems of exact and approximate controllability are to be distinguished. Exact controllability enables to steer the system to arbitrary final state while approximate controllability means that the system can be steered to arbitrary small neighborhood of final state. In particular, approximate controllable systems are more prevalent and very often approximate controllability is completely adequate in applications. There are many papers on the approximate controllability of the various types of nonlinear systems under different conditions (see \cite{pb1, jpd1, xf1, tg1, nim3, nim4, nim5, nim6, rs5, rs6, ns1, ns2, vv2, jw2, zy3, zy2} and references therein)

Recently, Chang \cite{ykc4} proved the controllability of semilinear mixed Volterra-Fredholm type integro-differential inclusions in Banach spaces by using Bohnenblust-Karlin's fixed point theorem. In \cite{jw2}, Wang et al. established the Existence and controllability results for fractional semilinear differential inclusions by using fractional calculation, operator semigroups and Bohnenblust-Karlin's fixed point theorem. In \cite{nim4}, Mahmudov proved the approximate controllability of fractional evolution equations by using of the theory of fractional calculus, semigroup theory and the Schauder's fixed point theorem. In \cite{ns1}, Sukavanam discussed the approximate controllability of a delayed semilinear control system with growing nonlinear term by using Schauder's fixed point theorem. Very recently, in \cite{rs6} Sakthivel et al. studied the approximate controllability of fractional nonlinear differential inclusions with initial and nonlocal conditions by using Bohnenblust-Karlin's fixed point theorem. In \cite{tg1}, Guendouzi investigated the approximate controllability for a class of fractional neutral stochastic functional integro-differential inclusions Bohnenblust-Karlin's fixed point theorem.

Motivated by the above works, this paper establishes a sufficient condition for the approximate controllability for a class of fractional semilinear integro-differential inclusions in Hilbert spaces of the form
\begin{eqnarray}
\begin{cases}
x'(t)\in\int_0^t \frac{(t-s)^{\alpha-2}}{\Gamma(\alpha-1)}Ax(s)ds+(Bu)(t)+F(t,x(t)),\quad t\in I=[0,b],  \label{c7i1} \\
x(0)=x_0,
\end{cases}
\end{eqnarray}
where $1<\alpha<2$, $A:D(A)\subset X\ra X$ is a linear densely defined operator of sectorial type on a Hilbert space $X$.  Notice that the convolution integral in the equation is known as the Riemann-Liouville fractional integral, the control function $u(\cdot)\in L^2(I,U)$, a Hilbert space of admissible control functions. Further, $B$ is a bounded linear operator from $U$ to $X$, and $F:I\times X\ra 2^X\setminus\{\emptyset\}$ is a nonempty, bounded, closed and convex multivalued map.

To the best of our knowledge, the study of the approximate controllability for a class of fractional semilinear integro-differential inclusions in Hilbert spaces of the form \eqref{c7i1}, is an untreated topic in the literature, this will be the main motivation of our paper. This paper is organized as follows. In Section 3, we establish a set of sufficient conditions for the approximate controllability for a class of fractional semilinear integro-differential inclusions in Hilbert spaces. In Section 4, we establish a set of sufficient conditions for the approximate controllability for a class of fractional semilinear integro-differential inclusions with nonlocal conditions. An example is presented in Section 5 to illustrate the theory of the obtained results.

\section{Preliminaries}
\noindent

Let $(Z,\|\cdot\|)$ and $(W,\|\cdot\|)$ be two Hilbert spaces. The notation $\mc{L}(Z,W)$ stands for the space of bounded linear operators from $Z$ into $W$ endowed with the uniform operator topology, and we abbreviate it to $\mc{L}(Z)$ whenever $Z=W$. In order to give an operator theoretical approach we recall the following definition (cf.\cite{cc3}).

\begin{definition}
Let $A$ be a closed and linear operator with domain $D(A)$ defined on a Hilbert space $X$. We recall that $A$ is the generator of a solution operator if there exist $\mu\in \mb{R}$ and a strongly continuous function $S_\alpha:\mb{R}_+\ra \mc{L}(X)$ such that $\{\lambda^\alpha:Re\lambda>\mu\}$ and
\begin{align*}
\lambda^{\alpha-1}(\lambda^\alpha-A)^{-1}x=\int_0^t e^{-\lambda t}S_\alpha(t)xdt,\quad Re\lambda>\mu,\quad x\in X.
\end{align*}
\end{definition}
In this case, $S_\alpha(t)$ is called the solution operator generated by $A$.

The concept of a solution operator, as defined above, is closely related to the concept of a resolvent family (see \cite[Chapter I]{jp1}). For the scalar case, there is a large bibliography, and we refer the reader to the monograph by Gripenberg et al. \cite{gg1}, and references therein. Because of the uniqueness of the Laplace transform, in the border case $\alpha=1$ the family $S_\alpha(t)$ corresponds to a $C_0$-semigroup, whereas in the case $\alpha=2$ a solution operator corresponds to the concept of a cosine family; see \cite{wa1, of1}. We note that solution operators, as well as resolvent families, are a particular case of $(a,k)$-regularized families introduced in \cite{cl1}. According to \cite{cl1} a solution operator $S_\alpha(t)$ corresponds to a $(1,\frac{t^{\alpha-1}}{\Gamma(\alpha)})$-regularized family. The following result is a direct consequence of \cite[Proposition 3.1 and Lemma 2.2]{cl1}.

\begin{proposition}
Let $S_\alpha(t)$ be a solution operator on $X$ with generator $A$. Then, we have
\begin{enumerate}
\item[$\bf(a)$] $S_\alpha(t)D(A)\subset D(A)$ and $AS_\alpha(t)x=S_\alpha(t)Ax$ for all $x\in D(A)$, $t\ge 0$.
\item[$\bf(b)$] Let $x\in D(A)$ and $t\ge 0$. Then $S_\alpha(t)x=x+\int_0^t\frac{(t-s)^{\alpha-1}}{\Gamma(\alpha)}AS_\alpha(s)ds$.
\item[$\bf(c)$] Let $x\in X$ and $t\ge 0$. Then $\int_0^t\frac{(t-s)^{\alpha-1}}{\Gamma(\alpha)}AS_\alpha(s)xds\in D(A)$ and
\begin{align*}
S_\alpha(t)x=x+A\int_0^t\frac{(t-s)^{\alpha-1}}{\Gamma(\alpha)}S_\alpha(s)xds
\end{align*}.
\end{enumerate}
\end{proposition}
A characterization of generators of solution operators, analogous to the Hille-Yosida Theorem for $C_0$-semigroup, can be directly deduced from \cite[Theorem 3.4]{cl1}. Results on perturbation, approximation, representation as well as ergodic type theorems can be deduced from the more general context of $(a,k)$-regularized resolvents (see \cite{cl1, ss1}).

A closed and linear operator $A$ is said to be sectorial of type $\mu$ if there exist $0<\te< \pi/2$, $\wt{M}>0$ and $\mu\in \mb{R}$ such that its resolvent exists outside the sector
\begin{align*}
\mu+S_\te=\{\mu+s:\lambda\in \mb{C},|arg(-\lambda)|<\te\}
\end{align*}
and
\begin{align*}
\|(\lambda-A)^{-1}\|\le \frac{\wt{M}}{|\lambda-\mu|},\lambda\notin \mu+S_\te.
\end{align*}
Sectorial operator are well studied in the literature. For a recent work including several examples and properties, we refer the reader to \cite{mh1}. In this work we will assume that in the problem \eqref{c7i1} the operator $A$ is sectorial of type $\mu$ with $0<\te< \pi(1-\alpha/2)$. Then $A$ is the generator of a solution operator given by
\begin{align*}
S_\alpha(t)=\frac{1}{2\pi i}\int_\gamma e^{\lambda t}\lambda^{\alpha-1}(\lambda^\alpha-A)^{-1}d\lambda,
\end{align*}
where $\gamma$ is a suitable path lying outside the sector $\mu+S_\te$ (cf. Cuesta's paper \cite{ec1}). Very recently, Cuesta \cite[Theorem 1]{ec1} has proved that if $A$ is a sectorial operator of type $\mu$ for some $\wt{M}>0$ and $0<\te< \pi(1-\alpha/2)$ then there exists $C>0$ such that, for all $t\ge 0$,
\begin{align*}
\|S_\alpha(t)\|_{\mc{L}(x)}=\frac{C\wt{M}}{1+|\mu|t^\alpha},\quad \mbox{if} \ \mu<0,
\end{align*}
and
\begin{align*}
\|S_\alpha(t)\|_{\mc{L}(x)}={C\wt{M}}({1+\mu t^\alpha})e^{\mu^{1/\alpha}t},\quad \mbox{if} \ \mu \ge 0,
\end{align*}
Note that $S_\alpha(t)$ is, in fact, integrable on $[0,b]$.

We also introduce some basic definitions and results of multivalued maps. For more details on multivalued maps, see the books of Deimling \cite{kde1} and Hu and Papageorgious \cite{sh1}.

A multivalued map $G: X\ra 2^X\setminus \{\emptyset\}$ is convex (closed) valued if $G(x)$ is convex (closed) for all $x\in X$. $G$ is bounded on bounded sets if $G(C)=\bigcup_{x\in C} G(x)$ is bounded in $X$ for any bounded set $C$ of $X$, i.e., $\sup_{x\in C}\Big\{\sup\{\|y\|:y\in G(x)\}\Big\}<\infty$.

\begin{definition}
$G$ is called upper semicontinuous (u.s.c. for short) on $X$ if for each $x_0\in X$, the set $G(x_0)$ is a nonempty closed subset of $X$, and if for each open set $C$ of $X$ containing $G(x_0)$, there exists an open neighborhood $V$ of $x_0$ such that $G(V)\subseteq C$.
\end{definition}

\begin{definition}
$G$ is called completely continuous if $G(C)$ is relatively compact for every bounded subset $C$ of $X$.
\end{definition}

If the multivalued map $G$ is completely continuous with nonempty values, then $G$ is u.s.c., if and only if $G$ has a closed graph, i.e., $x_n\ra x_*$, $y_n\ra y_*$, $y_n\in Gx_n$ imply $y_*\in Gx_*$. $G$ has a fixed point if there is a $x\in X$ such that $x\in G(x)$.

\begin{definition}
A function $x\in \mc{C}$ is said to be a mild solution of system \eqref{c7i1} if $x(0)=x_0$ and there exists $f\in L^1(I,X)$ such that $f(t)\in F(t,x(t))$ on $t\in I$ and the integral equation
\begin{align*}
x(t)=\mathcal{S}_\alpha(t)x_0+\int_0^t\mathcal{S}_\alpha(t-s)f(s)ds+\int_0^t\mathcal{S}_\alpha(t-s)Bu(s)ds,\  t\in I.
\end{align*}
is satisfied.
\end{definition}

In order to address the problem, it is convenient at this point to introduce two relevant operators and basic assumptions on these operators:

\begin{align*}
\Upsilon_0^b&=\int_0^b\mc{S}_\alpha(b-s)BB^*\mc{S}^*_\alpha(b-s)ds: X\ra X,\\
R(a,\Upsilon_0^b)&=(a I+\Upsilon_0^b)^{-1}: X\ra X
\end{align*}
where $B^*$ denotes the adjoint of $B$ and $\mc{S}_\alpha^*(t)$ is the adjoint of $\mc{S}_\alpha(t)$. It is straightforward that the operator $\Upsilon_0^b$ is a linear bounded operator.

To investigate the approximate controllability of system \eqref{c7i1}, we impose the following condition:
\begin{enumerate}
\item [$\bf H_0$] $aR(a,\Upsilon_0^b)\ra 0$ as $a\ra 0^+$ in the strong operator topology.
\end{enumerate}

In view of \cite{nim1}, Hypothesis $\bf{H_0}$ holds if and only if the linear fractional system
\begin{eqnarray}
x'(t)&=&\int_0^t \frac{(t-s)^{\alpha-2}}{\Gamma(\alpha-1)}Ax(s)ds+(Bu)(t),\quad t\in [0,b],  \label{c7p1} \\
\n x(0)&=&x_0
\end{eqnarray}
is approximately controllable on $[0,b]$.

\begin{lemma}\cite[Lasota and Opial]{al1}
Let $I$ be a compact real interval, $BCC(X)$ be the set of all nonempty, bounded, closed and convex subset of $X$ and $F$ be a multivalued map satisfying $F:I\times X\ra BCC(X)$ is measurable to $t$ for each fixed $x\in X$, u.s.c. to $x$ for each $t\in I$, and for each $x\in \mc{C}$ the set
\begin{align*}
S_{F,x}=\{f\in L^1(I,X):f(t)\in F(t,x(t)),\  t\in I\}
\end{align*}
is nonempty. Let $\ms{F}$ be a linear continuous from $L^1(I,X)$ to $\mc{C}$, then the operator
\begin{gather*}
\ms{F} \circ S_F:\mc{C}\ra BCC(\mc{C}),\ x\ra (\ms{F} \circ S_F)(x)=\ms{F}(S_{F,x}),
\end{gather*}
is a closed graph operator in $\mc{C}\times \mc{C}$.
\end{lemma}

\begin{lemma} \cite[Bohnenblust and Karlin]{hfb1}. \label{29}
Let $\mc{D}$ be a nonempty subset of $X$, which is bounded, closed, and convex. Suppose $G:\mc{D}\ra 2^X\setminus \{\emptyset\}$ is u.s.c. with closed, convex values, and such that $G(\mc{D})\subseteq \mc{D}$ and $G(\mc{D})$ is compact. Then $G$ has a fixed point.
\end{lemma}

\section{Controllability results}\label{s1}
\noindent

In this section, first we establish a set of sufficient conditions for the approximate controllability for a class of fractional semilinear integro-differential inclusions in Hilbert spaces by using Bohnenblust-Karlin's fixed point theorem. In order to establish the result, we need the following hypothesis:

\begin{enumerate}
\item [$\bf H_1$] The solution operator $\mc{S}_\alpha(t), {t\in I}$ is compact, and there exists $M>0$ such that $\|\mc{S}_\alpha(t)\|\le M$ for every $t\in I$.

\item[$\bf H_2$] For each positive number $r$ and $x\in \mc{C}$ with $\|x\|_\mc{C}\le r$, there exists a constant $q_1\in (0, q)$ and $L_{f,r}(\cdot)\in L^\frac{1}{q_1}(I,\mb{R}^+)$ such
that
\begin{gather*}
\sup\{\|f\|:f(t)\in F(t,x(t))\}\le L_{f,r}(t),
\end{gather*}
for a.e. $t\in I$.

\item[$\bf H_3$] The function $s\ra  L_{f,r}(s)\in L^1([0,t],\mb{R}^+)$ and there exists a $\gamma>0$ such that
\begin{gather*}
\lim_{r\ra \infty} \frac{\int_0^t L_{f,r}(s)ds}{r}=\gamma <+\infty.
\end{gather*}
\end{enumerate}

It will be shown that the system \eqref{c7i1} is approximately controllable, if for all $a>0$, there exists a continuous function $x(\cdot)$ such that
\begin{align}
x(t)&=\mc{S}_\alpha(t)x_0+\int_0^t \mc{S}_\alpha(t-s)f(s)ds+\int_0^t \mc{S}_\alpha(t-s)Bu(s,x)ds,\quad f\in S_{F,x},\label{c7m1}\\
u(t,x)&=B^*\mc{S}^*_\alpha(b-t)R(a,\Upsilon_0^b)p(x(\cdot)),\label{c7m2}
\end{align}
where $p(x(\cdot))=x_b-\mc{S}_\alpha(t)x_0-\int_0^t \mc{S}_\alpha(t-s)f(s)ds$.

\begin{theorem}\label{34}
Suppose that the hypotheses $\bf{H_0}$-$\bf{H_3}$ are satisfied. Then system \eqref{c7i1} controllable on $I$ provided that
\begin{align}
M\gamma\Big[1+\frac{1}{\alpha}M^{2}M_{B}^{2}b\Big]<1,\label{31}
\end{align}
where $M_B=\|B\|$.
\end{theorem}

\begin{proof}
The main aim in this section is to find conditions for solvability of system \eqref{c7m1} and \eqref{c7m2} for $a>0$. We show that, using the control $u(x,t)$, the operator ${\Gamma}:\mc{C}\ra 2^{\mc{C}}$, defined by
\begin{align*}
{\Gamma}(x)=\Big\{\varphi\in \mc{C}:\varphi (t)=\mc{S}_\alpha(t)x_0+\int_0^t \mc{S}_\alpha(t-s)[f(s)+Bu(s,x)]ds,\ f\in S_{F,x}\Big\},
\end{align*}
has a fixed point $x$, which is a mild solution of system \eqref{c7i1}. We observe that $x_b\in ({\Gamma}x)(b)$ which means that $u(t,x)$ steers system \eqref{c7i1} from $x_0$ to $x_b$ in finite time $b$. This implies system \eqref{c7i1} is controllable on $I$.

We now show that $\Gamma$ satisfies all the conditions of Lemma \ref{29}. For the sake of convenience, we subdivide the proof into five steps.

\noindent{\bf Step 1.} $\Gamma$ is convex for each $x\in \mc{C}$.

In fact, if $\varphi_1$, $\varphi_2$ belong to $\Gamma(x)$, then there exist $f_1$, $f_2\in S_{F,x}$ such that for each $t\in I$, we have
\begin{align*}
\varphi_i(t)=&\mc{S}_\alpha(t)x_0+\int_0^t \mc{S}_\alpha(t-s)f_i(s)ds+\int_0^t \mc{S}_\alpha(t-s)BB^{*}S_{\alpha}^{*}(b-t)
R(a,\Upsilon_0^b)\Big[x_b-\mc{S}(b)x_0 \\&-\int_0^b  \mc{S}_\alpha(b-\eta)f_i(\eta)d\eta\Big](s) ds,\quad i=1,2.
\end{align*}
Let $\lambda\in [0,1]$. Then for each $t\in J$, we get
\begin{align*}
\lambda\varphi_1(t)+(1-\lambda)\varphi_2(t)=&\mc{S}_\alpha(t)x_0+\int_0^t \mc{S}_\alpha(t-s)[\lambda f_1(s) +(1-\lambda)f_2(s)]ds\\&+\int_0^t \mc{S}_\alpha(t-s)BB^{*}S_{\alpha}^{*}(b-t) R(a,\Upsilon_0^b)\Bigg[x_b-\mc{S}(b)x_0\\&-\int_0^b \mc{S}_\alpha(b-s)[\lambda f_1(s) +(1-\lambda)f_2(s)]ds\Bigg](s)ds.
\end{align*}
It is easy to see that $S_{F,x}$ is convex since $F$ has convex values. So, $\lambda f_1+(1-\lambda)f_2\in S_{F,x}$. Thus,
$$\lambda\varphi_1+(1-\lambda)\varphi_2\in \Gamma(x).$$

\noindent{\bf Step 2.} For each positive number $r>0$, let $\mf{B}_r=\{x\in \mc{C}:\|x\|_\mc{C}\le r\}$. Obviously, $\mf{B}_r$ is a bounded, closed and convex set of $\mc{C}$. We claim that there exists a positive number $r$ such that $\Gamma(\mf{B}_r )\subseteq \mf{B}_r$.

If this is not true, then for each positive number $r$, there exists a function $x^r\in \mf{B}_r$, but $\Gamma(x^r)$ does not belong to $\mf{B}_r$, i.e.,
$$\|\Gamma(x^r)\|_\mc{C}\equiv\sup\Big\{\|\varphi^r\|_\mc{C}:\varphi^r\in (\Gamma x^r)\Big\}>r$$
and
\begin{align*}
\varphi^r(t)=\mc{S}_\alpha(t)x_0+\int_0^t \mc{S}_\alpha(t-s)f^r(s)ds+\int_0^t \mc{S}_\alpha(t-s)Bu^r(s,x) ds,
\end{align*}
for some $f^r\in S_{F,x^r}$. Using $\bf H_1$-$\bf H_3$, we have
\begin{align*}
r&<\|\Gamma(x^r)(t)\|\\&
\le\|\mc{S}_\alpha(t)x_0\|+\int_0^t \|\mc{S}_\alpha(t-s)f^r(s)\|ds+\int_0^t \|\mc{S}_\alpha(t-s)Bu^r(s,x)\|ds\\
&\le M\Big[\|x_0\|+\int_0^t L_{f,r}(s)ds\Big] +\frac{1}{\alpha}M^{2}M_{B}^{2}b\Bigg[\|x_b\|+\|x_0\|+\int_0^b L_{f,r}(s)ds\Bigg].
\end{align*}
Dividing both sides of the above inequality by $r$ and taking the limit as $r\ra \infty$, using $\bf H_3$, we get
\begin{gather*}
M\gamma\Big[1+\frac{1}{\alpha}M^{2}M_{B}^{2}b\Big]\ge 1.
\end{gather*}
This contradicts with the condition \eqref{31}. Hence, for some $r>0$, $\Gamma(\mf{B}_r )\subseteq \mf{B}_r$.

\noindent{\bf Step 3.} $\Gamma$ sends bounded sets into equicontinuous sets of $\mc{C}$. For each $x\in \mf{B}_r$, $\varphi\in\Gamma(x)$, there exists a $f\in S_{F,x}$ such that
\begin{align*}
\varphi(t)=\mc{S}_\alpha(t)x_0+\int_0^t \mc{S}_\alpha(t-s)f(s)ds+\int_0^t \mc{S}_\alpha(t-s)Bu(s,x)ds.
\end{align*}
Let $0<\ve<0$ and $0<t_1<t_2\le b$, then
\begin{align*}
|\varphi(t_1)-\varphi(t_2)|=&|\mc{S}_\alpha(t_1)-\mc{S}_\alpha(t_2)||x_0|+\Big|\int_{0}^{t_1-\ve}[\mc{S}_\alpha(t_1-s)-\mc{S}_\alpha(t_2-s)]f(s) ds\Big|\\& +\Big|\int_{t-\ve}^{t_1}[\mc{S}_\alpha(t_1-s)-\mc{S}_\alpha(t_2-s)]f(s) ds\Big|+\Big|\int_{t_1}^{t_2}\mc{S}_\alpha(t_2-s)f(s)ds\Big|\\& +\Big|\int_{0}^{t_1-\ve}[\mc{S}_\alpha(t_1-\eta)-\mc{S}_\alpha(t_2-\eta)]Bu(\eta,x) d\eta\Big|\\& +\Big|\int_{t-\ve}^{t_1}[\mc{S}_\alpha(t_1-\eta)-\mc{S}_\alpha(t_2-\eta)]Bu(\eta,x) d\eta\Big|+\Big|\int_{t_1}^{t_2}\mc{S}_\alpha(t_2-\eta)Bu(\eta,x) d\eta\Big|\\
\le &|\mc{S}_\alpha(t_1)-\mc{S}_\alpha(t_2)||x_0|+\int_{0}^{t_1-\ve}|\mc{S}_\alpha(t_1-s)-\mc{S}_\alpha(t_2-s)|L_{f,r}(s) ds\\& +\int_{t-\ve}^{t_1}|\mc{S}_\alpha(t_1-s)-\mc{S}_\alpha(t_2-s)|L_{f,r}(s) ds+M\int_{t_1}^{t_2}L_{f,r}(s) ds\\& +M_B\int_{0}^{t_1-\ve}|\mc{S}_\alpha(t_1-\eta)-\mc{S}_\alpha(t_2-\eta)|\|u(\eta,x)\| d\eta\\& +M_B\int_{t-\ve}^{t_1}|\mc{S}_\alpha(t_1-\eta)-\mc{S}_\alpha(t_2-\eta)\|u(\eta,x)\| d\eta+MM_B\int_{t_1}^{t_2}\|u(\eta,x)\| d\eta.
\end{align*}

The right-hand side of the above inequality tends to zero independently of $x\in B_r$ as $(t_1-t_2)\ra 0$ and $\ve$ sufficiently small, since the compactness of $\mc{S}_\alpha(t)$ implies the continuity in the uniform operator topology. Thus $\Gamma(x^r)$ sends $B_r$ into equicontinuous family of functions.

\noindent{\bf Step 4.} The set $\Pi(t)=\big\{\varphi(t):\varphi\in \Gamma(\mf{B}_r)\big\}$ is relatively compact in $X$.

Let $t\in (0,b]$ be fixed and $\ve$ a real number satisfying $0< \ve<t$. For $x\in B_r$, we define
\begin{align*}
\varphi_\ve(t)=\mc{S}_\alpha(t)x_0+\int_0^{t-\ve}\mc{S}_\alpha(t-s)f(s)ds+\int_0^{t-\ve}\mc{S}_\alpha(t-\eta)Bu(\eta,x)d\eta.
\end{align*}
Since $\mc{S}_\alpha(t)$ is a compact operator, the set $\Pi_\ve(t)=\{\varphi_\ve(t):\varphi_\ve\in \Gamma(B_r)\}$ is relatively compact in $X$ for each $\ve$, $0<\ve<t$. Moreover, for each $0<\ve<t$, we have
\begin{align*}
|\varphi(t)-\varphi_\ve(t)|\le M\int_{t-\ve}^tL_{f,r}(s)ds+MM_B\int_{t-\ve}^t\|u(\eta,x)\|d\eta.
\end{align*}
Hence there exist relatively compact sets arbitrarily close to the set $\Pi(t)=\{\varphi(t):\varphi\in \Gamma(B_r)\}$, and the set $\wt{\Pi}(t)$ is relatively compact in $X$ for all $t\in [0,b]$. Since it is compact at $t=0$, hence $\Pi(t)$ is relatively compact in $X$ for all $t\in [0,b]$.

\noindent{\bf Step 5.} $\Gamma$ has a closed graph.

Let $x_n\ra x_*$ as $n\ra \infty$, $\varphi_n\in {\Gamma}(x_n)$, and $\varphi_n\ra \varphi_*$ as $n\ra \infty$. We will show that $\varphi_*\in {\Gamma}(x_*)$. Since $\varphi_n\in {\Gamma}(x_n)$, there exists a $f_n\in S_{F,x_n}$ such that
\begin{align*}
\varphi_n(t)=&\mc{S}_\alpha(t)x_0+\int_0^t\mc{S}_\alpha(t-s)f_n(s)ds+\int_0^t \mc{S}_\alpha(t-s)BB^{*}S_{\alpha}^{*}(b-t)
R(a,\Upsilon_0^b)\Big[x_b-\mc{S}(b)x_0 \\&-\int_0^b  \mc{S}_\alpha(b-\eta)f_n(\eta)d\eta\Big](s)ds.
\end{align*}

We must prove that there exists a $f_*\in S_{F,x_*}$ such that
\begin{align*}
\varphi_*(t)=&\mc{S}_\alpha(t)x_0+\int_0^t \mc{S}_\alpha(t-s)f_*(s)ds+\int_0^t \mc{S}_\alpha(t-s)BB^{*}S_{\alpha}^{*}(b-t)\Big[x_1-\mc{S}(b)x_0\\&-\int_0^b \mc{S}_\alpha(b-\eta)f_*(\eta)ds\Big](s)ds.
\end{align*}
Set
\begin{gather*}
\ov{u}_x(t)=B^{*}S_{\alpha}^{*}(b-t)[x_b-\mc{S}(b)x_0](t).
\end{gather*}
Then
\begin{gather*}
\ov{u}_{x_n}(t)\ra \ov{u}_{x_*}(t), \quad \mbox{for} \ t\in I, \ \mbox{as} \ n\ra \infty.
\end{gather*}
Clearly, we have
\begin{align*}
\Big\|\Big(\varphi_n&-\mc{S}_\alpha(t)x_0-\int_0^t \mc{S}_\alpha(t-s)BB^{*}S_{\alpha}^{*}(b-t)
R(a,\Upsilon_0^b)\Big[x_b-\mc{S}(b)x_0 \\&-\int_0^b  \mc{S}_\alpha(b-\eta)f_n(\eta)d\eta\Big](s)ds\Big)-\Big(\varphi_*-\mc{S}_\alpha(t)x_0-\int_0^t \mc{S}_\alpha(t-s)BB^{*}S_{\alpha}^{*}(b-t)\Big[x_1\\&-\mc{S}(b)x_0-\int_0^b \mc{S}_\alpha(b-\eta)f_*(\eta)ds\Big](s)ds\Big)\Big\|_\mc{C} \ra 0 \ \mbox{as} \ n\ra \infty.
\end{align*}
Consider the operator $\wt{\ms{F}}:L^1(I,X)\ra \mc{C}$,
\begin{align*}
(\wt{\ms{F}}f)(t)=\int_0^t \mc{S}_\alpha(t-s)\Big[f(s)-BB^{*}S_{\alpha}^{*}(b-t)\Big(\int_0^b  \mc{S}_\alpha(b-\eta)f(\eta)d\eta\Big)(s)\Big]ds
\end{align*}

We can see that the operator $\wt{\ms{F}}$ is linear and continuous. From Lemma \ref{29} again, it follows that $\wt{\ms{F}}\circ S_F$ is a closed graph operator. Moreover,
\begin{align*}
\Big(\varphi_n&-\mc{S}_\alpha(t)x_0-\int_0^t \mc{S}_\alpha(t-s)BB^{*}S_{\alpha}^{*}(b-t)
R(a,\Upsilon_0^b)\Big[x_b-\mc{S}(b)x_0 \\&-\int_0^b  \mc{S}_\alpha(b-\eta)f_n(\eta)d\eta\Big](s)ds\Big)\in \ms{F}(S_{F,x_n}).
\end{align*}

In view of $x_n\ra x_*$ as $n\ra \infty$, it follows again from Lemma \ref{29} that
\begin{align*}
\Big(\varphi_*&-\mc{S}_\alpha(t)x_0-\int_0^t \mc{S}_\alpha(t-s)BB^{*}S_{\alpha}^{*}(b-t)
R(a,\Upsilon_0^b)\Big[x_b-\mc{S}(b)x_0 \\&-\int_0^b  \mc{S}_\alpha(b-\eta)f_*(\eta)d\eta\Big](s)ds\Big)\in \ms{F}(S_{F,x_*}).
\end{align*}
Therefore $\Gamma$ has a closed graph.

As a consequence of {\bf Steps 1-5} together with the Arzela-Ascoli theorem, we conclude that $\Gamma$ is a compact multivalued map, u.s.c. with convex closed values. As a consequence of Lemma \ref{29}, we can deduce that $\Gamma$ has a fixed point $x$ which is a mild solution of system \eqref{c7i1}.
\end{proof}

\begin{definition}
The control system \eqref{c7i1} is said to be approximately controllable on $I$ if for all $x_0\in X$, there is some control $u^2L(I,U)$, the closure of the reachable set, $\ov{R(b,x_0)}$ is dense in $X$, i.e., $\ov{R(b,x_0)}=X$, where $R(b,x_0)=\{x(b;u):u^2L(I,U),x(0,u)=x_0\}$ is the reachable set of system \eqref{c7i1} with the initial value $x_0$ at terminal time $b$.
\end{definition}

Roughly speaking, from any given starting point $x_0\in X$ we can go with the trajectory as close as possible to any other final state $x_b\in X$. In the following theorem, it will be shown that under certain conditions the approximate controllability of linear fractional inclusion \eqref{c7p1} implies the approximate controllability of the nonlinear fractional differential inclusion \eqref{c7i1}.

\begin{theorem}\label{40}
Suppose that the assumptions $\bf {H_0}$-$\bf {H_3}$ hold. Assume additionally that there exists $N\in L^1(I,[0,\infty))$ such that $\sup_{x\in X}\|F(t,x)\|\le N(t)$ for a.e. $t\in I$, then the nonlinear fractional differential inclusion \eqref{c7i1} is approximately controllable on $I$.
\end{theorem}

\begin{proof}
Let $\widehat{x}^a(\cdot)$ be a fixed point of $\Gamma$ in $\mf{B}_r$. By Theorem \ref{34}, any fixed point of $\Gamma$ is a mild solution of \eqref{c7i1} under the control
\begin{align*}
\widehat{u}^a(t)=B^*\mc{S}^*_\alpha(b-t)R(a,\Upsilon_0^b)p(\widehat{x}^a)
\end{align*}
and satisfies the following inequality
\begin{align*}
\widehat{x}^a(b)=x_b+aR(a,\Upsilon_0^b)p(\widehat{x}^a).\label{kd1}
\end{align*}
Moreover by assumption on $F$ and Dunford-Pettis Theorem, we have that the $\{f^a(s)\}$ is weakly compact in $L^1(I,X)$, so there is a subsequence, still denoted by $\{f^a(s)\}$, that converges weakly to say $f(s)$ in $L^1(I,X)$. Define
\begin{gather*}
w=x_b-\mc{S}_\alpha(b)x_0-\int_0^b  \mc{S}_\alpha(b-s)f(s)ds.
\end{gather*}
Now, we have
\begin{align}
\n \|p(\widehat{x}^a)-w\|=&\Big\|\int_0^b  \mc{S}_\alpha(b-s)[f(s,\widehat{x}^a(s))-f(s)]ds\Big\|\\
\le&\sup_{t\in J}\Big\|\int_0^t  \mc{S}_\alpha(t-s)[f(s,\widehat{x}^a(s))-f(s)]ds\Big\|.\label{kd2}
\end{align}
By using infinite-dimensional version of the Ascoli-Arzela theorem, one can show that an operator $l(\cdot)\ra \int_0^{\cdot} \mc{S}_\alpha(\cdot-s)l(s)ds:L^1(I,X) \ra C(I,X)$ is compact. Therefore, we obtain that $\|p(\widehat{x}^a)-w\|\ra 0$ as $a\ra 0^+$. Moreover, from \eqref{kd1} we get
\begin{align*}
\|\widehat{x}^a(b)-x_b\|\le & \|aR(a,\Upsilon_0^b)(w)\|+\|aR(a,\Upsilon_0^b)\|\|p(\widehat{x}^a)-w\|\\
\le&\|aR(a,\Upsilon_0^b)(w)\|+\|p(\widehat{x}^a)-w\|.
\end{align*}
It follows from assumption $\bf{H_0}$ and the estimation \eqref{kd2} that $\|\widehat{x}^a(b)-x_b\|\ra 0$ as $a\ra 0^+$. This proves the approximate controllability of differential inclusion \eqref{c7i1}.
\end{proof}

\section{Fractional control systems with nonlocal conditions}
\noindent

There exist an extensive literature of  differential equations with nonlocal conditions. The result concerning the existence and uniqueness of mild solutions to abstract Cauchy problems with nonlocal initial conditions was first formulated and proved by Byszewski, see \cite{lb1, lb2}. Since the appearance of this paper, several papers have addressed the issue of existence and uniqueness of nonlinear differential equations. Existence and controllability results of nonlinear differential equations and fractional differential equations with nonlocal conditions has been studied by several authors for different kind of problems \cite{jpc1, xf1, nim3, vv1, yz12}.

Recently, Mahmudov \cite{nim3} studied the approximate controllability of evolution systems with nonlocal conditions by usnig Schauder's fixed point theorem. In \cite{rs5} Sakthivel et al. discussed the approximate controllability of semilinear fractional differential systems with initial and nonlocal conditions by using Schauder's fixed point theorem.  Very recently Sakthivel et al. \cite{rs5} proved the exact controllability for a class of fractional-order neutral evolution control systems with initial and nonlocal conditions by using Contraction mapping principle. In \cite{vv1} Vijayakumar et al. established the nonlocal controllability of mixed Volterra-Fredholm type fractional semilinear integro-differential inclusions in Banach spaces by using Bohnenblust-Karlin's fixed point theorem and in \cite{vv3} investigated the controllability for a class of fractional neutral integro-differential equations with unbounded delay by using Contraction mapping principle.

Motivated by this consideration, in this section, we discuss the approximate controllability for a class of fractional integro-differential inclusions with nonlocal condition of the form
\begin{eqnarray}
\begin{cases}
x'(t)\in\int_0^t \frac{(t-s)^{\alpha-2}}{\Gamma(\alpha-1)}Ax(s)ds+Bu(t)+F(t,x(t)),\quad t\in I=[0,b],  \label{c7n1} \\
x(0)+g(x)=x_0,
\end{cases}
\end{eqnarray}where $g:C(I,X)\ra X$ is a given function which satisfies the following condition:

\begin{enumerate}
\item[$\bf H_4$] There exists a constant $L>0$ such that $|g(x)-g(y)|\le L\|x-y\|$, for $x,y\in C(I,X)$.
\end{enumerate}

The nonlocal term $g$ has a better effect on the solution and is more precise for physical measurements than the classical condition $x(0)=x_0$ alone. For example, $g(x)$ can be written as
\begin{gather*}
g(x)=\sum_{k=1}^m c_kx(t_k)
\end{gather*}
where $c_k(k=1,2,\cdots,n)$ are given constants and $0<t_1<\cdots<t_n\le b$.

\begin{definition}
A function $x\in \mc{C}$ is said to be a mild solution of system \eqref{c7n1} if $x(0)+g(x)=x_0$ and there exists $f\in L^1(I,X)$ such that $f(t)\in F(t,x(t))$ on $t\in I$ and the integral equation
\begin{align*}
x(t)=\mathcal{S}_\alpha(t)(x_0-g(x))+\int_0^t\mathcal{S}_\alpha(t-s)f(s)ds+\int_0^t\mathcal{S}_\alpha(t-s)Bu(s)ds,\  t\in I.
\end{align*}
is satisfied.
\end{definition}

\begin{theorem}\label{3.4}
Assume that the assumptions of Theorem \ref{34} are satisfied. Further, if the hypothesis $\bf H_4$ is satisfied, then the fractional system system \eqref{c7n1} is approximately controllable on $I$ provided that
\begin{align*}
M\gamma\Big[1+\frac{1}{\alpha}M^{2}M_{B}^{2}b\Big]<1,
\end{align*}
where $M_B=\|B\|$.
\end{theorem}

\begin{proof}
For each $a>0$, we define the operator $\widehat{\Gamma}_a$ on $X$ by
\begin{align*}
(\widehat{\Gamma}_a x)=z,
\end{align*}
where
\begin{align*}
z(t)&=\mc{S}_\alpha(t)(x_0-g(x))+\int_0^t \mc{S}_\alpha(t-s)f(s)ds+\int_0^t \mc{S}_\alpha(t-s)Bu(s,x)ds,\quad f\in S_{F,x},\\
v(t)&=B^*\mc{S}^*_\alpha(b-t)R(a,\Upsilon_0^b)p(x(\cdot)),\\
p(x(\cdot))&=x_b-\mc{S}_\alpha(t)(x_0-g(x))-\int_0^t \mc{S}_\alpha(t-s)f(s)ds
\end{align*}
It can be easily proved that if for all $a>0$, the operator $\widehat{\Gamma}_a$ has a fixed point by implementing the technique used in Theorem \ref{34}. Then, we can show that the fractional control system \eqref{c7n1} is approximately controllable. The proof of this theorem is similar to that of Theorem \ref{34} and Theorem \ref{40}, and hence it is omitted.
\end{proof}

\section{An example}
\noindent

As an application of our results we consider the following fractional differential inclusion of the form
\begin{align}
\n \frac{\partial u}{\partial t} x(t,\xi) &\in \frac{1}{\Gamma(\alpha-1)}\int_t^0(t-s)^{\mu-2}L_\xi u(s,\xi)ds\\&+[f_1(t,\varphi(0,\xi)),f_2(t,\varphi(0,\xi))]+\mu(t,\xi),\ t\in [0,b], \ \xi\in [0,\pi]\label{36}\\
x(t,0)=&x(t,\pi)=0,\label{38}\\
x(0,\xi)=&x_0(0,\xi),\quad \xi\in [0,\pi],\label{33}
\end{align}
where $1<\alpha<2$, $L$ stands for the operator with respect to the spatial variable $\xi$ which is given by
\begin{gather*}
L_\xi=\frac{\partial^2}{\partial \xi^2}-r, \quad \mbox{with} \quad r>0,
\end{gather*}
$f_1,f_2:I\times X\ra\mb{R}$  are measurable in $t$ and continuous in $y$. We assume that for each $t\in I$, $f_1(t,\cdot)$ is lower semicontinuous (i.e. the set $\{y\in X:f_1(t,y)>v\}$ is open for all $v\in \mb{R}$), and assume that for each $t\in I$, $f_2(t,\cdot)$ is u.s.c. (i.e. the set $\{y\in X:f_1(t,y)<v\}$ is open for  $v\in \mb{R}$).

Take $X=L^2([0,\pi],\mb{R})$ and the operator $A:L_\xi:D(A)\subset X\ra X$ with domain  $ D(A) = \{ x \in X :x^{\prime \prime} \in X, x(0) = x(\pi) = 0 \}$.
Clearly $A$ is densely defined in $X$ and is sectorial. Hence $A$ is a generator of a solution operator on $X$. The mutivalued map $F$ is u.s.c. with compact convex values \cite{kde1}.

By defining the following
\begin{align*}
y(t)(\xi)&=u(t,\xi)\\
F(t,\xi)&=[f_1(t,\varphi(0,\xi)),f_2(t,\varphi(0,\xi))]\\
Bu(t,\xi)&=\mu(t,\xi)
\end{align*}
we can transform \eqref{36}-\eqref{33} into the abstract form \eqref{c7i1}. Hence all the hypotheses of theorem \ref{40} are satisfied, then the system \eqref{36}-\eqref{33} is approximately controllable.

\end{document}